\documentclass{amsart}
\usepackage{amssymb}
\usepackage{amsthm}
\usepackage{indentfirst}
\usepackage{amsmath}
\usepackage{amscd}
\usepackage{xypic}

\numberwithin{equation}{section}

\newtheorem{Theorem}{Theorem}[section]
\newtheorem{Definition}[Theorem]{Definition}
\newtheorem{Proposition}[Theorem]{Proposition}
\newtheorem{Lemma}[Theorem]{Lemma}

\newtheorem{Assumption-Notation}[Theorem]{Assumption-Notation}

\newtheorem{Remark}[Theorem]{Remark}

\newtheorem{Problem}{Problem}

\newtheorem{Claim}[Theorem]{Claim}

\def\dim{\operatorname{dim}}

\begin{document}

\title[The cohomological support locus and the Iitaka fibration]{The cohomological support locus of pluricaonical sheaves and the Iitaka fibration}
\address{Lei Zhang\\College of Mathematics and information Sciences\\Shaanxi Normal University\\Xi'an 710062\\P.R.China}
\email{lzhpkutju@gmail.com}
\author{Lei Zhang}

\begin{abstract}
Let $\mathrm{alb}_X: X \rightarrow \mathrm{Alb}(X)$ be the Albanese map of a smooth projective variety and $f: X \rightarrow Y$ the fibration arising from the Stein factorization of $\mathrm{alb}_X$. For a positive integer $m$, if $f$ and $m$ satisfy two certain assumptions AS(1, 2), then the translates through the origin of all components of cohomological locus $V^0(\omega_X^m, \mathrm{alb}_X) \subset \mathrm{Pic}^0(X)$ generates $I^*\mathrm{Pic}^0(S)$ where $I: X \rightarrow S$ denotes the Iitaka fibration. As an application, we study pluri-canonical maps.
\end{abstract}

\maketitle

\section{Introduction}\label{1}
\textbf{Conventions:}
We work over complex numbers. For a smooth projective variety $X$ with $q(X) >0$, we usually denote by $\mathrm{alb}_X: X \rightarrow
\mathrm{Alb}(X)$ the Albanese map. A smooth variety $X$ with Kodaira dimension $\kappa(X) \geq 0$ has Iitaka fibration $I: X \dashrightarrow S$ which is defined by $|mK_X|$ for sufficiently divisible $m$. Blowing up $X$ if necessary, throughout this paper we always assume $I$ is a morphism. A \textbf{good minimal model} of a variety is a birational model with semi-ample canonical divisor and $\mathbb{Q}$-factorial canonical singularities. For a cartier divisor $D$ on a variety $X$ with $|D| \neq \emptyset$, we denote by $\phi_{|D|}$ the map induced by the linear system $|D|$.

Let $a: X \rightarrow A$ be a map from a projective variety to an abelian variety, and $\mathcal{F}$ a sheaf on $X$. The cohomological support locus of $\mathcal{F}$ w.r.t. $a$ is defined as
$$V^i(\mathcal{F}, a) := \{\alpha \in \mathrm{Pic}^0(A)|h^i(\mathcal{F} \otimes a^*\alpha) \neq 0\}$$
The cohomological support locus of the canonical sheaf $V^i(\omega_X, \mathrm{alb}_X)$ plays an important role in studying irregular varieties (see for example \cite{GL1}, \cite{GL2} and \cite{CH1}), which is closely related to pluri-canonical maps and Iitaka fibrations. Recall the results of \cite{CH3} and \cite{JS} on this topic.

\begin{Theorem}\label{chjs}
Let $X$ be a smooth irregular variety with $\kappa(X) \geq 0$. Denote by $I: X \rightarrow S$ be the Iitaka fibration.

(1) If $X$ is of maximal Albanese dimension (m.A.d.), then the subgroup of $\mathrm{Pic}^0(X)$
generated by the translates through the origin of the components of $V^0(\omega_X, \mathrm{alb}_X)$ is $I^*\mathrm{Pic}^0(S)$.

(2) If $X$ is of general type, and the Albanese fibers are of dimension 1, then  $V^0(\omega_X, \mathrm{alb}_X) = \mathrm{Pic}^0(X)$.
\end{Theorem}

The result above is applied to study the birationality of pluri-canonical maps. Let $X$ be a variety of general type with Albanese fibers being of dimension $\leq 1$. It is proved that for every $n \geq 4$ and $\alpha \in \mathrm{Pic}^0(X)$, $|\omega_X^n \otimes \alpha|$ induces a birational map (cf. \cite{Ti} and \cite{JS}); moreover if $X$ is of m.A.d. then $|\omega_X^3|$ induces a birational map (\cite{JLT}).

If Albanese fibers have higher dimension, then the cohomological locus $V^0(\omega_X, \mathrm{alb}_X)$ may contain little information (for example if a general fiber $F$ has $p_g(F) = 0$). We will study the cohomological support locus $V^0(\omega_X^m, \mathrm{alb}_X)$ of the pluri-canonical sheaf, which has been studied in \cite{CH2} and \cite{Lai}. To state our main result, we introduce three assumptions for a fibration $f:X \rightarrow Y$ between two smooth projective varieties and a positive integer $m$:

\begin{enumerate}
\item[AS(1):] A general fiber $F$ of $f$ has a good minimal model, and $|mK_F| \neq \emptyset$.
\item[AS(2):]
For every smooth projective curve $C$ on $Y$, if $X_C:=X \times_Y C$ is smooth and the fibration $f_C:= f|_{X_C}: X_C \rightarrow C$ is flat, then $\deg(f_{C*}\omega^m_{X_C/C})$ ($= \det (f_*\omega^m_{X/Y})\cdot C$ by adjunction formula and the flatness of $f_C$) is positive, unless there exists an \'{e}tale cover $\tilde{C} \rightarrow C$ such that $X_C \times_C \tilde{C}$ is birationally equivalent to  $F \times \tilde{C}$.
\item[WAS(2):]{With the notation as in AS(2), we have $\deg(f_{C*}\omega^m_{X_C/C}) > 0$, unless $f_C$ is birationally isotrivial, i.e., there exists a flat base change $\tilde{C} \rightarrow C$ such that $X_C \times_C \tilde{C}$ is birationally equivalent to $F \times \tilde{C}$.}
\end{enumerate}

Our main result is
\begin{Theorem}\label{main}
Let $X$ be a smooth projective variety with $\kappa(X) \geq 0$, and denote by $I: X \rightarrow S$ the Iitaka fibration. Let $f: X \rightarrow Y$ be the fibration arising from the Stein factorization of the Albanese map  $\mathrm{alb}_X: X \rightarrow \mathbf{Alb}(X)$. Suppose that the assumptions AS(1, 2) are satisfied for the fibration $f$ and the integer $m >0$. Then the subgroup of $\mathrm{Pic}^0(X)$ generated by the translates through the origin of the components of $V^0(\omega_X^m, \mathrm{alb}_X)$ is $I^*\mathrm{Pic}^0(S)$, moreover $q(S) =  q(X) -(\dim(X) - \kappa(X)) - (\dim(F) - \kappa(F))$.
\end{Theorem}

As an application, we prove
\begin{Theorem}\label{pcm}
Let $X$ be a smooth projective variety, $I: X \rightarrow S$ the Iitaka fibration and $a = \mathrm{alb}_S \circ I$. Let $f: X \rightarrow Y$ be the fibration arising from the Stein factorization of $a: X \rightarrow \mathrm{Alb}(S)$ and $F$ a general fiber.  Suppose that
\begin{itemize}
\item[(1)]{AS(1, 2) are satisfied for $f$ and $m =1$;}
\item[(2)]{for the integer $n \geq 2$, the pluri-canonical map $\phi_{|nK_F|}$ is birational to the Iitaka fibration of $F$.}
\end{itemize}
Then for every $\alpha \in \mathrm{Pic}^0(S)$, the map $\phi_{|(n+2)K_X \otimes I^*\alpha|}$ is birational to the Iitaka fibration.
\end{Theorem}
\begin{Remark}
In the theorem above, if assuming $m>1$ in assumption (1), the corresponding result is not better than that of \cite{CH4} Theorem 2.8.
\end{Remark}

For the assumptions AS(1, 2), we remark the following facts.

(1)
Assumption $AS(1)$ is needed when proving that $(\mathrm{alb}_X)_*\omega_X^m$ is a GV-sheaf. The existence of good minimal models is still a conjecture, it has been proved for the varieties of general type or of m.A.d.. Please refer to \cite{Lai} for the newest results.

(2)
If $f: X \rightarrow Y$ is fibred by curves of genus $\geq 2$ or points, then AS(1, 2) are satisfied for $f$ and $m=1$ (\cite{BPV} Chap. III Theorem 18.2). So Theorem \ref{chjs} is a special case of Theorem \ref{main}. For a smooth variety $X$ of general type, if the Albanese fibers are points or curves of genus $\geq 3$, applying Theorem \ref{pcm}, we get that $\phi_{|4K_X \otimes \alpha|}$ is birational for every $\alpha \in \mathrm{Pic}^0(X)$, which coincides with the results of \cite{Ti} and \cite{JS} Theorem 5.3 for the case $g=3$.

(3)
The assumption AS(2) is stronger than WAS(2). By \cite{BPV} Chap. III Theorem 18.2, the assumption WAS(2) is satisfied for an elliptic fibration and $m=1$, but AS(2) may fail. An example is constructed as follows. Take an elliptic curve $E$, a translate $t_a$ where $a \in E$ is a torsion point of order two and a curve $\tilde{C}$ granted an involution $\sigma$ with fixed point. Let $X = (E \times \tilde{C})/<t_a \times \sigma>$ and $C = \tilde{C}/\sigma$. For the natural fibration $f: X \rightarrow C$, we have $\deg(f_*\omega_{X/C}) = 0$ (see the proof of Lemma \ref{istfc1} or \cite{BPV} Chap. III Theorem 18.2), and AS(2) fails for $f$ and $m =1$ since $f$ has multiple fibers.

(4)
The assumption AS(2) is necessary for Theorem \ref{main}: in Section \ref{eg}, we construct a variety $X$ such that $\mathrm{alb}_X$ is an elliptic fibration and the conclusion of Theorem \ref{main} fails for $V^0(\omega_X, \mathrm{alb}_X)$.

Comparing the assumptions AS(2) and WAS(2), we have
\begin{Theorem}\label{istf}
Let $f: X \rightarrow Y$ be a fibration, $F$ a general fiber and $m$ a positive integer. Assume that AS(1) is satisfied for $f$ and $m$.
If either $m = 1$ and the automorphism group $\mathrm{Aut}(F)$ acts faithfully on $H^0(F, \omega_F)$, or $m>1$,
then  WAS(2) is equivalent to AS(2).
\end{Theorem}

Recall that for a fibration $f: X \rightarrow C$ from a smooth projective variety to a smooth curve, it is known that $f_*\omega_{X/C}^m$ is weakly positive (\cite{Vie2} Thmeorem III), hence $\deg(f_*\omega^m_{X/C})\geq 0$; and with the assumption that a general fiber $F$ has a good minimal model $F_0$, there exists an integer $m >1$, which depends on an $m_0$ such that $|m_0K_{F_0}|$ is base point free (\cite{Ka} Theorem 1.1, Lemma 6.1 and Sec. 7), such that if $\deg(f_*\omega^m_{X/C})= 0$ then $f$ is birationally isotrivial. Therefore, for a fibration with fibers having good minimal models, there always exists $m$ such that the assumptions AS(1, 2) hold.

It is of special interest to consider the problem whether (W)AS(2) is true for certain fibrations and $m=1$, which relates to the numerical criterion for whether a fibration is birationally isotrivial.
We propose the following problem
\begin{Problem}
Let $f: X \rightarrow C$ be a fibration to a curve with general fibers having good minimal models and $\deg(f_*\omega_{X/C}) = 0$. Is the fibration $f$ birationally isotrivial under one of the following conditions?

(1) general fibers are of maximal Albanese dimension;

(2) for a general fiber $F$ the Cox ring $\oplus_{m = 0}^{m= +\infty} H^0(F, mK_F)$ is generated by $H^0(F, K_F)$.
\end{Problem}

The numerical criterion works for the fibrations fibred by curves of genus $g \geq 1$ which are of type (1). We prove that it also works for the fibrations fibred by varieties with trivial canonical bundle  which are of type (2).

\begin{Theorem}[Theorem \ref{K0}]
Let $X$ be a smooth projective variety and $f: X \rightarrow C$ a fibration to a smooth projective curve. Suppose that a general fiber $F$ of $f$ has good minimal models with trivial canonical bundles. If $\deg(f_*\omega_{X/C}) = 0$, then $f$ is birationally isotrivial.

If moreover the automorphism group $\mathrm{Aut}(F)$ acts faithfully on $H^0(F, \omega_F)$, then there exists an \'{e}tale cover $\pi: \tilde{C} \rightarrow C$ such that $X \times_C \tilde{C}$ is birational to $F \times \tilde{C}$.
\end{Theorem}

\section{Preliminaries}\label{tool}
\subsection{Fourier-Mukai transform}
If $A$ denotes an abelian variety, then $\hat{A}$ denotes its dual $\mathrm{Pic}^0(A)$, $\mathcal{P}$ denotes the Poincar\'{e} line bundle on
$A \times \hat{A}$, and the Fourier-Mukai transform $R\Phi_{\mathcal{P}}: D^b(A) \rightarrow D^b(\hat{A})$ w.r.t. $\mathcal{P}$ is defined as
$$R\Phi_{\mathcal{P}}(\mathcal{F}) := R(p_2)_*(Lp_1^*\mathcal{F} \otimes \mathcal{P})$$
where $p_1$ and $p_2$ are the projections from $A \times \hat{A}$ to $A$ and $\hat{A}$ respectively.
Similarly $R\Psi_{\mathcal{P}}: D^b(\hat{A}) \rightarrow D^b(A)$ is defined as
$$R\Psi_{\mathcal{P}}(\mathcal{F}) := R(p_1)_*(Lp_2^*\mathcal{F} \otimes \mathcal{P})$$

\begin{Theorem}[\cite{Mu}, Thm. 2.2]\label{Mu} Let $A$ be an abelian variety of dimension $d$. Then
$$R\Psi_{\mathcal{P}} \circ R\Phi_{\mathcal{P}} = (-1)_A^*[-d]~\mathrm{and}~R\Phi_{\mathcal{P}} \circ R\Psi_{\mathcal{P}} = (-1)_{\hat{A}}^*[-d].$$
\end{Theorem}

\subsection{GV-Sheaves, M-regular sheaves and CGG}
\begin{Definition}[\cite{PP2} Def. 2.1, 2.2, 5.2]
Given a coherent sheaf $\mathcal{F}$ on an abelian variety $A$, its $i^{\mathrm{th}}$ \emph{cohomological support locus} is defined as
$$V^i(\mathcal{F}): = \{\alpha \in \hat{A}| h^i(\mathcal{F} \otimes \alpha) > 0\}$$
The number $\mathrm{gv}(\mathcal{F}): = \mathrm{min}_{i>0}\{\mathrm{codim}_{\hat{A}}V^i(\mathcal{F}) - i\}$ is called the \emph{generic vanishing index} of $\mathcal{F}$, and
we say $\mathcal{F}$ is a \emph{GV-sheaf} if $\mathrm{gv}(\mathcal{F}) \geq 0$, and is \emph{M-regular} if $\mathrm{gv}(\mathcal{F}) > 0$, and is an \emph{$IT^0$-sheaf} if $V^i(\mathcal{F})=\emptyset$ for $i>0$.

We say $\mathcal{F}$ is \emph{continuously~ globally~ generated} (CGG) if the summation of the evaluation maps
$$\mathrm{ev}_U: \oplus_{\alpha \in U}H^0(\mathcal{F} \otimes \alpha) \otimes \alpha^{-1} \rightarrow \mathcal{F}$$
is surjective for any non-empty open set $U \subset \hat{A}$.
\end{Definition}

\begin{Theorem}[\cite{Ha} Thm. 1.2, Cor. 3.2 or \cite{PP2} Sec. 2]\label{btool}
Let $A$ be an abelian variety of dimension $d$ and $\mathcal{F}$ a GV-sheaf on $A$. Then
\begin{enumerate}
\item[(i)]{$\mathrm{Pic}^0(A)\supset V^0(\mathcal{F}) \supset V^1(\mathcal{F}) \supset ...\supset V^d(\mathcal{F})$, and $\mathcal{F} = 0$ if $V^0(\mathcal{F}) = \emptyset$;}
\item[(ii)]{$R\Phi_{\mathcal{P}}(R \Delta (\mathcal{F}))[d] \cong R^d\Phi_{\mathcal{P}}(R \Delta (\mathcal{F}))$ is a sheaf (where $R \Delta (\mathcal{F}) := R\mathcal{H}om(\mathcal{F},\mathcal{O}_A)$), which we denote by $\widehat{R \Delta (\mathcal{F})}$;}
\item[(iii)]{$\mathcal{E}xt_{\hat{A}}^i(\widehat{R \Delta (\mathcal{F})}, \mathcal{O}_{\hat{A}}) \cong (-1)_{\hat{A}}^*R^i\Phi_{\mathcal{P}}(\mathcal{F})$;}
\item[(iv)]{one direct summand of $\mathcal{F}$ is also a GV-sheaf.}
\end{enumerate}
\end{Theorem}

\begin{Proposition}[\cite{PP2} Cor. 5.3]\label{cgg}
An M-regular sheaf on an abelian variety is CGG.
\end{Proposition}

\subsection{Good minimal models and the cohomological support locus of $\omega_X^m$}
On good minimal models recall that
\begin{Proposition}[\cite{Lai} Thm. 4.5]\label{gm}
Let $a: X \rightarrow A$ be a morphism from a smooth projective variety to an abelian variety. Assume that general fibers of $a$ have good minimal models. Then $X$ has a good minimal model.
\end{Proposition}

\begin{Proposition}\label{gv}
Let $a: X \rightarrow A$ be a morphism from a smooth projective variety to an abelian variety. Suppose that $X$ has a good minimal model.
Then for any torsion line bundle $P \in \mathrm{Pic}(X)$ on $X$ (i.e. $nP \equiv \mathcal{O}_X$ for some integer $n>0$), the sheaf $a_*(\omega_X^m \otimes P)$ is a GV-sheaf on $A$, and every component of $V^i(a_*\omega_X^m),i\geq 0$ is a translate of a sub-torus of $\hat{A}$ via a torsion point.
\end{Proposition}
\proof Note that the argument of \cite{Lai} applies for a general map to an abelian variety not necessarily the Albanese map. By \cite{Lai} Theorem 3.5, it is known that every component of $V^i(a_*\omega_X^m),i\geq 0$ is a translate of a sub-torus of $\hat{A}$ via a torsion point. We still need to prove that $a_*(\omega_X^m \otimes P)$ is a GV-sheaf on $A$.

If $P \neq \mathcal{O}_X$, then consider the cyclic \'{e}tale cover induced by $nP \equiv \mathcal{O}_X$ where $n$ is the order of $P$. Note that $X'$ also has a good minimal model, and the sheaf $\omega_X^m \otimes P$ is a direct summand of $\pi_*\omega_{X'}^m$, so is $a_*(\omega_X^m \otimes P)$ of $a_*\pi_*\omega_{X'}^m$. To prove that $a_*(\omega_X^m \otimes P)$ is a GV-sheaf, by considering the map $a\circ \pi: X' \rightarrow A$ instead, we can assume $P = \mathcal{O}_X$.

Up to some blowing up maps, we can assume there is a morphism $\mu: X \rightarrow \bar{X}$ to one of its good minimal models. Since $\bar{X}$ has at most canonical singularities which hence are rational, the map $a$ factors through $\mu$.
Write that $a = \bar{a} \circ \mu$. Since $\mu_*\omega_X^m = \omega_{\bar{X}}^m$, we have
$$\clubsuit: a_*\omega_X^m = \bar{a}_*\omega_{\bar{X}}^m.$$

Again up to some blowing up maps, we can assume
\begin{itemize}
\item
$K_{X} = \mu^*K_{\bar{X}} + \sum_i a_iE_i$ where $E_i$ are the $\mu$-exceptional components and $a_i \geq 0$;
\item
there exist a sufficiently divisible integer $N>m$ and a smooth divisor $D \in |\mu^*NK_{\bar{X}}|$ by Bertini's theorem, such that $D + \sum_iE_i$ is a simple normal crossing divisor.
\end{itemize}
Then we have
$$\mu^*mK_{\bar{X}} = mK_X - \sum_ima_iE_i \leq  mK_X - \sum_i[(m - 1)a_i]E_i = mK_X - \sum_i[(m - 1)a_i]E_i - [\frac{m-1}{N}]D.$$
Note that $\mathcal{O}_{X}(- [\frac{m-1}{N}D + \sum_i(m-1)a_iE_i])$ is contained in the multiplier ideal sheaf $\mathcal{I}(||(m-1)K_X||)$.
So for an ample divisor $H$ on $A$ and $t>0$, we have the natural inclusions
$$\mu^*\omega_{\bar{X}}^m \subseteq \omega_{X}^m\otimes\mathcal{I}(||(m-1)K_X||) \subseteq \omega_{X}^m\otimes\mathcal{I}(||(m-1)K_X + \frac{1}{t}a^*H||) \subseteq  \omega_{X}^m$$
Pushing forward via $a_*$, we obtain that
$$\bar{a}_*\omega_{\bar{X}}^m \subseteq a_*(\omega_{X}^m\otimes\mathcal{I}(||(m-1)K_X + \frac{1}{t}\tilde{a}^*H||)) \subseteq  a_*\omega_{X}^m $$
Therefore, the two inclusions are equalities by $\clubsuit$. Then  $a_*\omega_{X}^m$ is a GV-sheaf by \cite{Lai} Lemma 3.4.
\qed

\subsection{Results on fibrations}\label{fbr}
Let $f: X \rightarrow Y$ be a fibration. The variation $\mathrm{Var}(f)$, roughly speaking, is the number of moduli of fibers of $f$ in the sense of birational geometry. Please refer to \cite{Ka} Sec. 1 for a precise definition. If $\mathrm{Var}(f)=0$, we say $f$ \emph{birationally isotrivial}; and if general fibers are isomorphic to each other, we say $f$ \emph{isotrivial}.

Here we recall the following results, which will be used in the sequel.

(1)
Let $X$ and $X'$ be two birational projective varieties with at most canonical singularities, and let $f: X \rightarrow Y$ and $f': X' \rightarrow Y$ be two birational fibrations. Then by considering a common resolution and comparing the push-forward of the pluri-canonical sheaves, we conclude that $f_*\omega_X^m \cong f'_*\omega_{X'}^m$.

(2)
Let $f: X \rightarrow Y$ be a birationally isotrivial fibration. Assume a general fiber $F$ is of Kodaira dimension $\kappa(F) \geq 0$ and has a good minimal model $\bar{F}$. Then there exist a birational model $\bar{f}: \bar{X} \rightarrow Y$ with good general fibers, and an \emph{equivariant resolution} $\nu: X' \rightarrow \bar{X}$ (i.e., there exists a lifting $\nu^*T_{\bar{X}} \rightarrow T_{X'}$ where $T_{\bar{X}}:= \mathcal{E}xt^0(\Omega_{\bar{X}}^1, \mathcal{O}_{\bar{X}})$, see \cite{Ka} p.14). By \cite{Ka} Lemma 7.1 and Corollary 7.3, both $\bar{f}$ and $f'$ are isotrivial. There exists a Galois cover $Z \rightarrow Y$ with Galois group $G$ and an action of $G$ on $\bar{F}$, such that $X$ is birational to $(\bar{F}\times Z)/G$ where $G$ acts on $\bar{F}\times Z$ diagonally. Denote by $H \vartriangleleft G$ the subgroup acting trivially on $\bar{F}$. Then $(\bar{F}\times Z)/H \cong \bar{F} \times (Z/H)$. By replacing the cover $Z \rightarrow Y$ by $Z/H \rightarrow Y$, we can assume $G$ acts on $\bar{F}$ faithfully.

\subsection{Results on Iitaka fibration}\label{itk}
Let $X$ be a smooth variety with $\kappa(X) \geq 0$, and denote by $I: X \rightarrow S$ the Iitaka fibration. Let $g: X \rightarrow Z$ be a fibration and $G$ a general fiber. Then
\begin{itemize}
\item[(1)]{every fiber of the Iitaka fibration of $G$ is contained in some fiber of $I: X \rightarrow S$;}
\item[(2)]{if $g$ factors through $I$, then $\kappa(X) = \kappa(G) + \dim Z$;}
\item[(3)]{if $\kappa(X) = 0$, then its Albanese map is a fibration onto its Albanese variety (\cite{CH1} Theorem 1);}
\item[(4)]{if $X$ has a good minimal model, then every Iitaka fiber has non-positive Kodaira dimension.}
\end{itemize}
Here we make a simple explanation for $(1,2,4)$. Up to some blowing up maps we get a fibration $h:X \rightarrow W$, such that $g$ factors through $h$ and that the restriction map $h|_G$ coincides with its Iitaka fibration of $G$ for general fiber $G$. Let $F'$ be a general fiber of $h$. Assume $K_X \sim_{\mathbb{Q}}I^*H + V$ where $H$ is an ample $\mathbb{Q}$-divisor on $S$. By $K_{F'} \sim K_X|_{F'}$, we have $|m(I^*H + V)||_{F'} \subseteq |mK_{F'}|$ for any $m > 0$. Since $\kappa(F') = 0$, $I^*H|_{F'}\sim_{\mathbb{Q}} 0$. This implies $F'$ is contained in some fiber of $I: X \rightarrow S$, hence $(1)$ follows. If $g$ factors through $I$, then $(1)$ implies that the Iitaka fibration of $G$ coincides with $I|_G$, thus $\kappa(G) = \dim I(G)$, so $(2)$ follows.
For $(4)$, take a good minimal model which is also denoted by $X$ to save notation. Assume that $K_X \sim_{\mathbb{Q}} I^*H$ where $H$ is an ample $\mathbb{Q}$-divisor on $S$. Then for every fiber $F$, $K_{F} \sim_{\mathbb{Q}} K_X|_{F} \sim_{\mathbb{Q}} I^*H|_{F}$ is numerically trivial, thus $F$ has non-positive Kodaira dimension.

\section{The main theorem}\label{4}

In this section we will prove Theorem \ref{main} and \ref{pcm}. Let the notation be as in Theorem \ref{main}

\subsection{The fibration induced by the cohomological support locus of $\omega_X^m$}
Here we denote by $\hat{T}$ the subgroup of $\mathrm{Pic}^0(X)$
generated by the translates through the origin of the components of $V^0(\omega_X^m, \mathrm{alb}_X)$, and denote by $\iota:\mathrm{Alb}(X) \rightarrow T$ the dual map of the inclusion $\hat{T} \hookrightarrow \mathrm{Pic}^0(X)$. So we can assume $V^0(\omega_X^m, \mathrm{alb}_X) \subset \tau + \hat{T}$ where $\tau$ consists of finite torsion points on $\mathrm{Pic}^0(X)$ by Proposition \ref{gv}. Let $\mathrm{alb}_Z \circ g: X\rightarrow Z \rightarrow T$ be the Stein factorization of $\iota\circ \mathrm{alb}_X$. Denote by $G$ a general fiber of $g$, by $K$ a fiber of $\iota$, and by $a_G$ the restriction map of $\mathrm{alb}_X$ on $G$. Then we get a commutative diagram
\[\begin{CD}
G     @> >>      X       @>g>>          Z \\
@Va_G VV          @V\mathrm{alb}_X VV                 @V\mathrm{alb}_Z VV \\
K     @> >>      \mathrm{Alb}(X)      @>\iota>>       T
\end{CD} \]

\begin{Claim}
$\dim(V^0(a_{G*}\omega_G^m)) = 0$.
\end{Claim}
\begin{proof}
Consider the quotient group homomorphism $\pi: \mathrm{Pic}^0(X) \rightarrow\hat{K}$ whose kernel contains some translates of $\hat{T}$ via some torsion points. Observe that
\begin{itemize}
\item[(a)]{by definition, for $\alpha \in \mathrm{Pic}^0(X)$, $V^0(\omega_X^m\otimes \alpha, \mathrm{alb}_X) + \alpha = V^0(\omega_X^m, \mathrm{alb}_X)$, thus $ V^0((\mathrm{alb}_Z \circ g)_*(\omega_X^m \otimes \alpha))= \hat{T}\cap V^0(\omega_X^m\otimes \alpha, \mathrm{alb}_X) \cong (\hat{T} + \alpha)\cap V^0(\omega_X^m\otimes \alpha, \mathrm{alb}_X)$;}
\item[(b)]{$V^0(\omega_X^m\otimes \alpha, \mathrm{alb}_X)$ is projected to finite torsion points $\pi(\tau)$ on $\hat{K}$ via $\pi: \mathrm{Pic}^0(X) \rightarrow \hat{K}$.}
\end{itemize}
For any torsion point $\alpha_0 \in \hat{K}$ not contained in $\pi(\tau)$, taking a torsion point $\alpha \in \pi^{-1}\alpha_0$, by (a,b) above we have $ V^0((\mathrm{alb}_Z \circ g)_*(\omega_X^m \otimes \alpha)) = \emptyset$. Since $(\mathrm{alb}_Z \circ g)_*(\omega_X^m \otimes \alpha)$ is a GV-sheaf by Proposition \ref{gv}, we conclude that $(\mathrm{alb}_Z \circ g)_*(\omega_X^m \otimes \alpha) = 0$ by Theorem \ref{btool} (i), thus $V^0(a_{G*}\omega_G^m \otimes \alpha_0) = 0$. Applying Proposition \ref{gv} to $G$, we conclude that $V^0(a_{G*}\omega_G^m) \subset \pi(\tau)$.
\end{proof}

The following proposition describes $G$, and the proof is postponed.
\begin{Proposition}\label{G}
There exists an \'{e}tale cover $\tilde{K} \rightarrow K$ such that the fiber product $G \times_K \tilde{K}$ is birational to $F \times \tilde{K}$. In particular the Iitaka fiber of $G$ is mapped onto $K$ via $\mathrm{alb}_X$.
\end{Proposition}

\subsection{$\hat{T} \subset  I^*\mathrm{Pic}^0(S)$}
Let $I: X \rightarrow S$ be the Iitaka fibration and $G'$ a fiber. We get a commutative diagram
\[\begin{CD}
G'     @> >>      X       @>I>>          S \\
@Va_{G'} VV          @V\mathrm{alb}_X VV                 @V\mathrm{alb}_S VV \\
K'     @> >>      \mathrm{Alb}(X)      @>\iota'>>       T'=\mathrm{Alb}(S)
\end{CD} \]
Each Iitaka fiber $G'$ has Kodaira dimension $\leq 0$, thus $V^0(a_{G'*}\omega_{G'}^m) \subset \hat{K'}$ is composed of at most finite torsion points, so there exists a set of torsion points $\tau' \subset \hat{K'}$ containing $V^0(a_{G'*}\omega_{G'}^m)$ for every $G'$. Consider the quotient map $\pi': \mathrm{Pic}^0(X) \rightarrow\hat{K'}$ whose kernel consists of some translates of $\iota'^*\hat{T'}$ via some torsion points. For any $\alpha_0 \in \hat{K'}$ not contained in $\tau'$, we have $V^0(a_{G'*}\omega_{G'}^m \otimes \alpha_0) = \emptyset$ and thus $h^0(G', \omega_{G'}^m\otimes \alpha_0) = 0$; we conclude that for every $\alpha \in \pi'^{-1}\alpha_0$, $\mathrm{alb}_{S*}I_*(\omega_X^m \otimes \alpha) = 0$, thus
$$\pi'^{-1}\alpha_0 \cap V^0(\omega_X^m, \mathrm{alb}_X) = \emptyset.$$
This means that $V^0(\omega_X^m, \mathrm{alb}_X)$ is projected into $\tau'$ on $\hat{K'}$ via $\pi': \hat{A} \rightarrow \hat{K'}$, thus the kernel contains $\hat{T}$, i.e., $\hat{T} \unlhd \iota'^*\hat{T'}$, consequently the dual map $T' \rightarrow T$ is a surjection.

\subsection{Proof of Theorem \ref{main}}
Consider the Iitaka fibration $I_G: G \rightarrow Z''$ and take a general fiber $F''$. Note that  $F''$ is contained in some $G'$ (see Sec. \ref{itk}). By Proposition \ref{G}, $F''$ is mapped onto $K$ via $\mathrm{alb}_X$, so $K \subseteq K'$ up to a translate. Combining that $T' \rightarrow T$ is a surjection, we conclude that $K' = K$ and $T' = T$, and thus $\hat{T} = \iota'^*\hat{T'} = I^*\mathrm{Pic}^0(S)$.

We still need to calculate $\dim(\hat{T})$.
Observe that Iitaka fibration factors through $g: X \rightarrow Z$. By the results of Sec. \ref{itk}, we have $\kappa(X) = \dim(Z) +  \kappa(G)$. On the other hand, by Proposition \ref{G} $\kappa(G) = \kappa(F)$. Then we are done by
\begin{equation}
\begin{split}
\dim(\hat{T})& = \dim(T) \\
& = q(X) - \dim(K) \\
& = q(X) - (\dim(G) - \dim(F)) \\
& = q(X) - (\dim(X) - \dim(Z) -\dim(F))\\
& = q(X) - (\dim(X) - \kappa(X) + \kappa(G) -\dim(F)) \\
& = q(X) -(\dim(X) - \kappa(X)) - (\dim(F) - \kappa(F))
\end{split}
\end{equation}

\subsection{Proof of Proposition \ref{G}} We break the proof into three steps.

Step 1: $a_{G*}\omega_G^m$ is a numerically trivial vector bundle on $K$.

Obviously $a_G = \pi \circ f|_G: G \rightarrow W \rightarrow K$ (where $W = f(G)$) coincides with the Stein factorization of $a_G$,  thus the assumptions AS(1, 2) hold for $f|_G$ and $m$.

Note that $a_{G*}\omega_G^m$ is a GV-sheaf by Proposition \ref{gv} and $\dim V^0(a_{G*}\omega_G^m) = 0$. Then using Theorem \ref{btool} we have
$$R\Phi_\mathcal{P}(R \triangle a_{G*}\omega_G^m) \cong R^d\Phi_\mathcal{P}(R \triangle a_{G*}\omega_G^m) [-d] = \tau[-d]$$
where $\tau$ is a coherent sheaf on $\hat{K}$ and $d = \dim K$; and
$$R^i\Phi_\mathcal{P} a_{G*}\omega_G^m\cong (-1)_{\hat{K}}^*\mathcal{E}xt^i(R^d\Phi_\mathcal{P}(R \triangle a_{G*}\omega_G^m), \mathcal{O}_{\hat{K}}) \cong (-1)_{\hat{K}}^*\mathcal{E}xt^i(\tau, \mathcal{O}_{\hat{K}})$$
is supported on at most finite points for every $i \geq 0$. It follows that $\dim(\mathrm{Supp}\tau) = 0$, thus $\mathcal{E}xt^i(\tau, \mathcal{O}_{\hat{K}}) = 0$ except when $i = d$. So we have
$$R\Phi_\mathcal{P}a_{G*}\omega_G^m \cong R^d\Phi_\mathcal{P} a_{G*}\omega_G^m[-d]\mathrm{~and~} \dim(\mathrm{Supp} R^d\Phi_\mathcal{P} a_{G*}\omega_G^m)=0.$$
Therefore, using Theorem \ref{Mu}, we conclude that
$$a_{G*}\omega_G^m \cong (-1)_K^*R\Psi_\mathcal{P} R\Phi_\mathcal{P} a_{G*}\omega_G^m[d] \cong (-1)_K^*R^0\Psi_\mathcal{P} R^d\Phi_\mathcal{P} a_{G*}\omega_G^m$$
is a numerically trivial vector bundle on $K$.

Step 2: $W$ is isomorphic to an abelian variety $K'$ \'{e}tale over $K$.

Take a smooth curve $i: B = H_1 \cap H_2 \cap ...\cap H_{d-1} \hookrightarrow K$ where $H_i$  are some general very ample divisors on $K$, such that both $C = \pi^{-1}B$ and $G_C = G \times_WC$ are smooth, and the projection $G_C = G \times_WC \rightarrow C$ is flat. We use $a_C$ and $f_C$ for the restriction maps of $a_G$ and $f$ on $G_C$. We have that $a_{C*}\omega_{G_C/B}^m \cong i^*a_{G*}\omega_{G/K}^m$ is numerically trivial, hence
$$\deg(a_{C*}\omega_{G_C/B}^m) = 0.$$

Let $r = h^0(F, \omega_F^m)$. Applying Riemann-Roch formula gives that
\begin{equation}
\begin{split}
&\chi(C, f_{C*}\omega_{G_C/B}^m) = \deg(f_{C*}\omega_{G_C/B}^m) + r(1-g(C))\\
&\chi(B, a_{C*}\omega_{G_C/B}^m) \\
&= \deg(a_{C*}\omega_{G_C/B}^m) + \deg(\pi)r(1-g(B))\\
& = \deg(\pi)r(1-g(B))\\
& = r(1-g(C)) + \frac{r}{2}\deg(\omega_{C/B})
\end{split}
\end{equation}
where the $3^{\mathrm{rd}}$ ``='' is due to $\deg(a_{C*}\omega_{G_C/B}^m ) = 0$, and the $4^{\mathrm{th}}$ ``='' is due to $2g(C) - 2 = \deg(\pi)(2g(B) - 2) + \deg(\omega_{C/B})$. Then by $\chi(C, f_{C*}\omega_{X_C/B}^m) = \chi(B, \pi_*f_{C*}\omega_{G_C/B}^m) = \chi(B, a_{C*}\omega_{G_C/B}^m)$, we get
$$\heartsuit:\deg(f_{C*}\omega_{X_C/B}^m)= \frac{r}{2}\deg(\omega_{C/B}).$$
On the other hand, we have
$$\deg(f_{C*}\omega_{G_C/B}^m) =\deg(f_{C*}\omega_{G_C/C}^m  \otimes \omega_{C/B}^m)= \deg(f_{C*}\omega_{G_C/C}^m)+ rm\deg(\omega_{C/B}).$$
Since $f_{C*}\omega_{G_C/C}^m$ is nef, thus $\deg(f_{C*}\omega_{G_C/B}^m) \geq rm\deg(\omega_{C/B})$, Then by $\heartsuit$ it is only possible that
$$\deg(\omega_{C/B}) = 0~\mathrm{and}~\deg(f_{C*}\omega_{X_C/C}^m) = 0,$$
thus $C \rightarrow B$ is \'{e}tale. If $D \subset K$ denotes the branch divisor of $\pi$, then $B \cdot D = 0$, thus $D = 0$ since $B$ is a complete intersection of ample divisors. We conclude that $\pi: W \rightarrow K$ is \'{e}tale in codimension 1, thus $\kappa(W) = 0$, which means $W$ is birational to an abelian variety $K'$ which is \'{e}tale over $K$ (see for example \cite{Ka0} Cor. 2). Since $\pi: W \rightarrow K$ is finite, it must be that $W =K'$.

Step 3: There is an \'{e}tale cover $\tilde{K} \rightarrow K'$ such that $G\times_{K'}\tilde{K}$ is birational to $F \times \tilde{K}$.

Take a good minimal model $\bar{G} \rightarrow K'$, and let $\bar{F}$ be a general fiber. By the assumption AS(2) and the results of Sec. \ref{fbr}, we have that $\bar{G}_C$ is isotrivial over $C$, thus $\bar{G}$ is isotrivial over $K'$. We can find a Galois cover $Z' \rightarrow K'$ with Galois group $H$ and a faithful action of $H$ on $\bar{F}$, such that $\bar{G}$ is birational to $(\bar{F} \times Z')/H$ where $H$ acts on $\bar{F} \times Z'$ diagonally. For the fibration $(\bar{F} \times Z')/H \rightarrow K'$, the fibers over the branch locus of $Z' \rightarrow K'$ are multiple fibers. By the assumption AS(2) again, $\bar{G}_C \rightarrow C$ has a birational model $\bar{G}'_C \rightarrow C$ such that every fiber is isomorphic to $\bar{F}$. We conclude that $C$ does not intersect the branch locus of $Z' \rightarrow K'$, thus $Z' \rightarrow K'$ is unramified in codimension 1, hence $Z'$ is birational to an abelian variety $\tilde{K}$. Then we finish the proof.

\subsection{Proof of  Theorem \ref{pcm}}\label{pm}
Applying the argument above to the map $a: X \rightarrow \mathrm{Alb}(S)$, we can prove that the translates through the origin of the components of $V^0(\omega_X, a)$ generates $\mathrm{Pic}^0(S)$.
Since $X$ has a good minimal model, using the proof of Proposition \ref{gv}, we have $a_*(\omega_{X}^k\otimes\mathcal{I}(||(k-1)K_X||)) = a_*(\omega_{X}^k) \neq 0$ for $k>0$.
Then
\begin{itemize}
\item[(I)]{For $k \geq 2$, we have $h^i(\mathrm{Alb}(S), a_*(\omega_{X}^{k} \otimes\mathcal{I}(||(k-1)K_X||))\otimes Q) = 0$ for any $Q \in \mathrm{Pic}^0(S)$ and $i>0$ by \cite{J} Lemma 4.2, the sheaf $a_*(\omega_{X}^{k})$ is an $IT^0$ sheaf, hence is CGG by Proposition \ref{cgg};}
\item[(II)]{Writing that $I_x \otimes \omega_X^k=I_x \otimes \omega_X \otimes \omega_X^{k-1}$, then arguing similarly as in \cite{Ti} or \cite{JS} Proposition 5.2, we can show that for general $x\in X$ the sheaf $a_*(I_x \otimes \omega_X^k)$ is $M$-regular for any $k\geq 2$, hence is CGG.}
\end{itemize}
Let $\Gamma$ be the support of the cokernal of $a^*a_*(\omega_{X}) \rightarrow \omega_{X}$. Take two distinct points $x, y \in X$ not contained in $\Gamma$ but separated by the Iitaka fibration. Suppose that $x$ is general. We will prove that the linear system $|(n+2)K_X \otimes I^*\alpha|$ separates $x$ and $y$. It suffices to find a section $s \in H^0(X, I_x \otimes \omega_X^{n+2}\otimes I^*\alpha) \neq 0$ not vanishing at $y$.

We will apply the proof of \cite{JS} Theorem 5.3.
By the assumption (2), the following evaluation map is surjective
$$a^*a_*(I_x\otimes \omega_X^n)\rightarrow a_*(I_x\otimes \omega_X^n) \rightarrow a_*(I_x\otimes \omega_X^n)|_y \cong \mathbb{C}(y).$$
Then by (II), for general $\beta \in \mathrm{Pic}^0(S)$ there exists $t_{-\beta} \in H^0(X, I_x\otimes\omega_X^n \otimes I^*\beta^{-1})$ not vanishing at $y$. Similarly, by (I), for general $\beta \in \mathrm{Pic}^0(S)$ there exists $s_{\beta} \in H^0(X, \omega_X^2\otimes I^*\alpha \otimes I^*\beta)$ not vanishing at $y$. Then the section $t_{-\beta}\otimes s_{\beta} \in H^0(X, I_x \otimes \omega_X^{n+2}\otimes I^*\alpha)$ does not vanish at $y$. We finish the proof.

\section{Isotrivial fibration with 0-degreed push-forward of relative pluri-canonical sheaves}

In this section we will study the isotrivial fibration $f: X \rightarrow B$ to a curve with $\deg(f_*\omega_{X/B}^m) = 0$, and get a relation between the two assumptions AS(2) and WAS(2).
Throughout this section, for a projective morphism $h:Y\rightarrow X$ between two varieties, we use $\omega_{Y/X}$ for the relative dualizing sheaf. For duality theory and some basic properties of relative dualizing sheaves, we refer to \cite{Har} Chap. III Sec. 8, 10, Chap. V Sec. 9 and Chap. VII Sec. 4 or \cite{Vie1} Sec. 6. In particular the relative dualizing sheaf $\omega_{Y/X}$ commutes with flat base changes and coincides with the relative canonical sheaf if $X$ and $Y$ are smooth varieties. First recall the following results
\begin{Lemma}[\cite{Vie2} Lemma 2.1]\label{vr}
Let $X$ be a smooth quasi-projective variety, $Y$ a normal variety and $h: Y \rightarrow X$ a finite morphism. Assume the discriminant $\Delta(Y/X) \subset X$ is a normal crossing divisor. Take a desingularization $d: Z \rightarrow Y$. Then

(1) $h$ is flat and $Y$ has only rational singularities;

(2) for an effective $d$-exceptional divisor $E$, the dualizing sheaf $\omega_Y \cong d_*\omega_Z(E)$;

(3) $\omega_X$ is a direct factor of $h_*d_*\omega_Z$.

\end{Lemma}

Let $\pi_0: C \rightarrow B=C/G$ be a Galois cover between two smooth curves. Denote by $R_{\pi_0}$ the ramification locus. Let $F$ be a smooth variety with a faithful action of $G$, $\bar{Y}= F \times C$, $\bar{X} = (F \times C)/G$ where $G$ acts on $F \times C$ diagonally, $\mu: X \rightarrow \bar{X}$ a smooth resolution, $\bar{\pi}: \bar{Y} \rightarrow \bar{X}$ the quotient map, $g: Y \rightarrow C$ the projection to $C$ and $f: X \rightarrow B$ the natural fibration.

Consider the following commutative diagram
\[\begin{CD}
Y''     @> \sigma >> Y'     @> \eta>> Y=X\times_B C  @> \pi >>  X \\
@Vg'' VV            @Vg' VV                @Vg VV         @Vf VV   \\
C     @> \mathrm{id}_C >>C     @> \mathrm{id}_C >> C  @> \pi_0 >>  B
\end{CD} \]
where $\eta:Y'\rightarrow Y$ is the normalization, and $\sigma: Y'' \rightarrow Y'$ is a smooth resolution. Let $\pi':=\pi \circ \eta: Y' \rightarrow X$ and $\pi''= \pi' \circ \sigma: Y'' \rightarrow Y$. We assume $\pi'$ is branched along a normal crossing divisor on $X$.

By construction, there is a birational morphism $\nu: Y'' \rightarrow \bar{Y}$, and we can assume the action of $G$ lifts on $Y''$. Then $\omega_{Y''}$, $\omega_C$ and $\omega_F$ are all $G$-invariant sheaves, and $G$ acts on their global sections via pulling back $\zeta \mapsto \zeta^*$ where $\zeta \in G$.

\begin{Lemma}\label{gai}
$\bar{\pi}^G_* \omega_{\bar{Y}} \cong \mu_*\omega_X$ where $\bar{\pi}^G_* \omega_{\bar{Y}}$ denote the $G$-invariant subsheaf of $\bar{\pi}_* \omega_{\bar{Y}}$.
\end{Lemma}
\begin{proof}
By Lemma \ref{vr} (3), we know that $\omega_X \cong \pi''^G_* \omega_{Y''}$. Then the lemma follows by $\nu_*\omega_{Y''} = \omega_{\bar{Y}}$.
\end{proof}

\begin{Lemma}\label{istfc1}
Assume that $r_1 =p_g(F) > 0$. Then $\deg(f_*\omega_{X/B}) = 0$ if and only if for every $P \in R_{\pi_0}$, the subgroup $G_P \vartriangleleft G$ fixing $P$, acts trivially on $H^0(F, \omega_F)$.
\end{Lemma}
\begin{proof}
Note that $G_P$ is cyclic. Assume $G_P$ is of order $k_P$. Take a generator $\sigma$ of $G_P$ and a local parameter $t$ of $C$ at $P$. We can assume $\sigma^*t = \xi t$ where $\xi$ is a $k_P^{\mathrm{th}}$ primitive root of $1$. We can choose a basis $s_i, i = 1,2, ..., r_1$ of $H^0(F, \omega_F)$ such that $\sigma^*s_i = \xi^{n_{P,i}} s_i$ where $0 \leq n_{P,i} \leq k_P-1$. Since $\pi_0: C \rightarrow B$ is a Galois cover, by Lemma \ref{gai} we have $f_*\omega_X \otimes \mathcal{O}_{\pi_0(P),B} \cong (g_*\omega_{\bar{Y}} \otimes \mathcal{O}_{P,C})^{G_P}$, so $t^{k_P - n_{P,i} - 1} dt \wedge s_i, i = 1,2, ..., r_1$ form a basis of $f_*\omega_X \otimes \mathcal{O}_{\pi_0(P),B}$.

Inversing the process above, we can take a basis $\alpha_i, i=1,2,...,r_1$ of $f_*\omega_X \otimes \mathcal{O}_{\pi_0(P), B}$ such that $t^{-(k_P - n_{P,i} - 1)}\pi^*\alpha_i, i = 1,2, ..., r_1$ form a basis of $g_*\omega_{\bar{Y}} \otimes \mathcal{O}_{P, C}$.
We conclude that
$$\deg(g_*\omega_{\bar{Y}}) = \deg(\pi_0^*\det (f_*\omega_{X})) + \sum_{P \in R_{\pi_0}}\sum_{i =1}^{i=r_1} (k_P - n_{P,i} - 1),$$
thus
\begin{equation}
\begin{split}
\deg(g_*\omega_{Y/C}) &= \deg(\pi_0^*\det (f_*\omega_{X})) + \sum_{P \in R_{\pi_0}}\sum_{i =1}^{i=r_1} (k_P - n_{P,i} - 1) - r_1\deg (\omega_C)\\
&= \deg(\pi_0^*\det (f_*\omega_{X/B})) + r_1\deg (\pi_0^*\omega_B) + \sum_{P \in R_{\pi_0}}\sum_{i =1}^{i=r_1} (k_P - n_{P,i} - 1) - r_1\deg (\omega_C)\\
& =  \deg (\pi_0)\deg(f_*\omega_{X/B}) + \sum_{P \in R_{\pi_0}}\sum_{i =1}^{i=r_1} (k_P - n_{P,i} - 1) - r_1\deg (\omega_{C/B})\\
& =  \deg (\pi_0)\deg(f_*\omega_{X/B}) + \sum_{P \in R_{\pi_0}}\sum_{i =1}^{i=r_1} (- n_{P,i})
\end{split}
\end{equation}

We can see that $\deg(\pi_0^*\det (g_*\omega_{Y/C})) = \deg(f_*\omega_{X/B})  =0$ if and only if $n_{P,i}= 0$ for every $P$ and $i$, which is equivalent that for every $P \in R_{\pi_0}$, $G_P$ acts trivially on $H^0(F, \omega_F)$.
\end{proof}

\begin{Lemma}\label{istfc2}
Assume that $m >1$ and $r_m = P_m(F)> 0$. Then $\deg(f_*\omega_{X/B}^m) = 0$ if and only if $\pi_0: C \rightarrow B$ is an \'{e}tale cover.
\end{Lemma}
\begin{proof}
The direction ``if'' is easy. We focus on the other direction, and assume $\deg(f_*\omega_{X/B}^m) = 0$.

We argue by contrary. Let $P$ be a ramification point of $\pi_0$, $F_P$ the fiber of the fibration $\bar{Y} \rightarrow C$ over $P$, $F''_P$ the  fiber of the fibration $Y'' \rightarrow C$ over $P$ and $\tilde{F}''_P$ the strict transform of $F_P$ via $\nu$. Denote by $E_i$ the $\nu$-exceptional components. Write that $K_{Y''} \sim \nu^*K_{\bar{Y}} + \sum_ia_iE_i$ and $\nu^*F_P =  \tilde{F}''_P + \sum_ib_iE_i$ where $a_i>0$, $b_i \geq 0$ and $b_i = 0$ if $E_j$ does not intersect $\tilde{F}''_P$. Here we claim that
$$\clubsuit: \sum_ia_iE_i + \tilde{F}''_P \geq \nu^*F_P.$$
Indeed, adjunction formula gives
$$K_{\tilde{F}''_P } \sim (K_{Y''} + \tilde{F}''_P)|_{\tilde{F}''_P } \sim(\nu^*(K_{\bar{Y}} + F_P) + \sum_i(a_i - b_i)E_i)|_{\tilde{F}''_P }.$$
Since $F_P$ is smooth, we have that $(a_i - b_i)\geq 0$, thus $\sum_ia_iE_i  + \tilde{F}''_P \geq \nu^*F_P$.

Since $\pi_0: C \rightarrow B$ is flat, we have $\omega_{Y/C} \cong \pi^*\omega_{X/B}$. From the finite morphism $\eta: Y' \rightarrow Y$, we have the trace map
$$\alpha_0: \eta_*\omega_{Y'/C} \rightarrow \omega_{Y/C},$$ and the pull-back homomorphism
$$\alpha_1: \eta^*\eta_*\omega_{Y'/C} \rightarrow \eta^*\omega_{Y/C}.$$
The natural homomorphism $ \eta^*\eta_*\omega_{Y'/C}\rightarrow \omega_{Y'/C}$ being surjective (since $\eta$ is affine) and $\omega_{Y/C}$ being invertible, the homomorphism $\alpha_1$ factors through a homomorphism
$$\alpha_2: \omega_{Y'/C} \rightarrow \eta^*\omega_{Y/C}.$$
By the isomorphism $\sigma_*\omega_{Y''/C} \cong \omega_{Y'/C}$ (Lemma \ref{vr} (2)), we get
$$\sigma^*\omega_{Y'/C} \cong \sigma^*\sigma_*\omega_{Y''/C} \rightarrow \omega_{Y''/C}$$
which is isomorphic outside the exceptional locus.
So for some effective $\sigma$-exceptional divisor $E$, the pull-back $\sigma^*\alpha_2$ induces an injection
$$\alpha_3: \omega_{Y''/C} \rightarrow \sigma^*\eta^*\omega_{Y/C}(E) \cong \pi''^*\omega_{X/B}(E).$$
Note that $\alpha_3$ is not necessarily surjective. Since $\pi'': Y'' \rightarrow X$ is unramified along $\tilde{F}''_P$ while $\pi_0: C \rightarrow B$ is ramified along $P$, the image of $\alpha_3$ is contained in $\mathcal{O}_{Y''}(-\tilde{F}''_P) \otimes \pi''^*\omega_{X/B}(E)$.
So we get a homomorphism
$$\alpha_3': \omega_{Y''/C}(\tilde{F}''_P) \rightarrow \sigma^*\eta^*\omega_{Y/C}(E) \cong \pi''^*\omega_{X/B}(E).$$
By $\clubsuit$, we have an injection $\nu^*\omega_{\bar{Y}/C}(F''_P) \rightarrow \omega_{Y''/C}(\tilde{F}''_P)$, and get the following homomorphism by composing this injection with $\alpha_3'$
$$\alpha_4: \nu^*\omega_{\bar{Y}/C}(F''_P) \rightarrow \sigma^*\eta^*\omega_{Y/C}(E) \cong \pi''^*\omega_{X/B}(E).$$
Tensoring $\alpha_4^{m-1}$ with $\mathrm{id}_{\omega_{Y''/C}}$ gives
$$\alpha_5: \nu^*\omega_{\bar{Y}/C}^{m-1}((m-1)F''_P) \otimes \omega_{Y''/C} \rightarrow \pi''^*\omega_{X/B}^{m-1} \otimes \omega_{Y''/C} ((m-1)E).$$
Applying $\sigma_*$ to $\alpha_5$, by projection formula and Lemma \ref{vr} (2), gives
$$\alpha_6: \sigma_*(\nu^*\omega_{\bar{Y}/C}^{m-1}((m-1)F''_P) \otimes \omega_{Y''/C}) \rightarrow \eta^*\pi^*\omega_{X/B}^{m-1} \otimes \sigma_*\omega_{Y''/C} \cong \eta^*\pi^*\omega_{X/B}^{m-1} \otimes \omega_{Y'/C}$$
Similarly applying $\eta_*$ to $\alpha_7$, then composing with $\mathrm{id}_{\pi^*\omega_{X/B}^{m-1}}\otimes \alpha_0$ gives
$$\alpha_7: \eta_*\sigma_*(\nu^*\omega_{\bar{Y}/C}^{m-1}((m-1)F''_P) \otimes \omega_{Y''/C}) \rightarrow \pi^*\omega_{X/B}^{m-1} \otimes \eta_*\omega_{Y'/C} \rightarrow \pi^*\omega_{X/B}^{m-1} \otimes \omega_{Y/C} \cong \pi^*\omega_{X/B}^{m}.$$
Applying $g_*$ to $\alpha_6$ we obtain the injection
\begin{equation}
\begin{split}
g_*\eta_*\sigma_*(\nu^*\omega_{\bar{Y}/C}^{m-1}((m-1)F''_P)\otimes\omega_{Y''/C}) &\cong \bar{g}_*\nu_*(\nu^*(\omega_{\bar{Y}/C}^{m-1}\otimes \bar{g}^*\mathcal{O}_{C}((m-1)P)) \otimes \omega_{Y''/C}) \\
&\cong \oplus^{r_m} \mathcal{O}_{C}((m-1)P) \rightarrow g_*\pi^*\omega_{X/B}^{m} \cong \pi_0^*f_*\omega_{X/B}^{m}.
\end{split}
\end{equation}
where the last ``$\cong$'' is due to that the base change $\pi_0: C \rightarrow B$ is flat.
A contradiction follows by $\deg(f_*\omega_{X/B}^m) = 0$.
\end{proof}

\begin{proof}[Proof of Theorem \ref{istf}]
Let $C, X_C, f_C$ be as in AS(2), and assume WAS(2). By the results of Sec. \ref{fbr}, there exists a finite cover $\pi_0: \tilde{C} \rightarrow C$ such that $X_C \times_C \tilde{C}$ is birational to $F \times \tilde{C}$. We can assume $\pi_0$ is a Galois cover with Galois group $G$, which acts on $F$ faithfully, such that $X_C$ is birational to $(F \times \tilde{C})/G$ where $G$ acts on $F \times \tilde{C}$ diagonally.

Let $X_C'$ be a smooth resolution of $(F \times \tilde{C})/G$ and $f_C': X_C'\rightarrow C$ the natural fibration. Then $f_{C*} \omega_{X_C/C}^m \cong f'_{C*} \omega_{X'_C/C}^m$. Applying Lemmas \ref{istfc1} and \ref{istfc2} to the fibration $f_C'$, we can conclude Theorem \ref{istf}.
\end{proof}

\section{The fibrations with general fibers having trivial canonical bundles}

\textbf{Notations and assumptions:}
Let $X$ be a smooth projective variety, $f: X \rightarrow C$ a fibration to a smooth projective curve and $n$ the dimension of general fibers. Suppose that general fibers have good minimal models with trivial canonical bundles.

\begin{Theorem}\label{K0}
If $\deg(f_*\omega_{X/C}) = 0$, then $f$ is birationally isotrivial.
\end{Theorem}

Before the proof, we introduce two lemmas.
\begin{Lemma}\label{bsch}
Let $C' \rightarrow C$ be a finite morphism between two smooth curves, let $X'$ be a resolution of the fiber product $X \times_C C'$, and denote by $f':X' \rightarrow C'$ the natural fibration. If $\deg(f_*\omega_{X/C}) = 0$, then $\deg(f'_*\omega_{X'/C'}) = 0$.
\end{Lemma}
\begin{proof}
This follows from \cite{Ka} Corollary 5.4.
\end{proof}

Using \cite{Ka1} Theorem 3 and the notation there, we have
\begin{Lemma}\label{bigness}
Let $U \subset C$ be an open set, $H_0$ a variation of Hodge structure with unipotent local monodromies and $H$ the extension of $H_0$ on $C$. If for general point $t \in C$ the natural homomorphism $T_{C,t} \rightarrow \mathrm{Hom}(F^{n,0}, F^{n-1,1})$ is injective, then $\deg(F^{n,0}) > 0$.
\end{Lemma}

\begin{proof}[Proof of Theorem \ref{K0}]
The fibration $f: X \rightarrow C$ has a birational model $\bar{f}: \bar{X} \rightarrow C$ such that general fiber $\bar{F}$ is a good minimal model. Replace $X$ by an equivariant resolution $\mu: X \rightarrow \bar{X}$. Then for the fiber $F = \mu^*\bar{F}$ over $\bar{F}$, the restriction map $\mu: F \rightarrow \bar{F}$ is an equivariant resolution, where the pull-back homomorphism $\mu^*T_{\bar{F}} \rightarrow T_F$ is induced as follows
\[\begin{CD}
\mu^*T_{\bar{F}}    @> >>    \mu^*T_{\bar{X}} \\
@V VV          @V VV               \\
T_F   @> >>     T_X
\end{CD} \]
In particular the following composite homomorphism is injective
$$\mathrm{Ext}^1(\Omega_F^1, \mathcal{O}_F)=H^1(F, T_{F})  \rightarrow H^1(\bar{F}, T_{\bar{F}}) \rightarrow \mathrm{Ext}^1(\Omega_{\bar{F}}^1, \mathcal{O}_{\bar{F}}).$$

We only need to prove that for general $t \in C$ the Kodaira-Spencer map is zero
$$\lambda_t: T_{t, C} \rightarrow H^1(F_t, T_{F_t})$$
where $F_t$ is the fiber over $t$.

With the help of Lemma \ref{bsch}, up to a base change we can assume that over an open set $U \subset C$, the natural variation of the Hodge structure on $R^nf_*\mathbb{C}$ has unipotent local monodromies. Using Lemma \ref{bigness}, by assumption that $\deg(f_*\omega_{X/C}) = 0$, the following composite map must be zero
$$\delta_t \circ\lambda_t: T_{t, C} \rightarrow H^1(F_t, T_{F_t}) \rightarrow \mathrm{Hom}(H^0(F_t, \Omega^n_{F_t}), H^1(F_t, \Omega^{n-1}_{F_t}))$$
where $\delta_t: H^1(F_t, T_{F_t}) \rightarrow \mathrm{Hom}(H^0(F_t, \Omega^n_{F_t}), H^1(F_t, \Omega^{n-1}_{F_t}))$ is the period map (induced by the cup product).
So we reduce to prove that for a general fiber $F$, the period map below is injective
$$\delta: H^1(F, T_{F}) \rightarrow \mathrm{Hom}(H^0(F, \Omega^n_{F}), H^1(F, \Omega^{n-1}_{F})).$$
(This is well known if $F$ is a smooth variety with trivial canonical bundle, please refer to \cite{GHJ} Sec. 16.2 whose argument applies for any dimension.)

\textbf{Step 1:}  By $\omega_{\bar{F}} = \mathcal{O}_{\bar{F}}$, the following homomorphism induced by the cup product is injective
$$\alpha: \mathrm{Ext}^1(\Omega_{\bar{F}}^1, \mathcal{O}_{\bar{F}}) \rightarrow \mathrm{Hom}(H^0({\bar{F}},\omega_{\bar{F}}), \mathrm{Ext}^1(\Omega_{\bar{F}}^1, \omega_{\bar{F}})).$$

\textbf{Step 2:}
Since $\bar{F}$ has at most rational singularities, by the projection formula we have $R\mu_*L\mu^*\Omega_{\bar{F}}^1\cong \Omega_{\bar{F}}^1 \otimes R\mu_*\mathcal{O}_F \cong \Omega_{\bar{F}}^1$. So using Grothendieck duality yields
$$R\mathrm{Hom}(L\mu^*\Omega_{\bar{F}}^1, \omega_F) \cong R\mathrm{Hom}(L\mu^*\Omega_{\bar{F}}^1, \mu^!\omega_{\bar{F}}) \cong R\mathrm{Hom}(\Omega_{\bar{F}}^1, \omega_{\bar{F}}),$$
thus $\mathrm{Ext}^1(L\mu^*\Omega_{\bar{F}}^1, \omega_F) \cong \mathrm{Ext}^1(\Omega_{\bar{F}}^1, \omega_{\bar{F}})$.
And using Grothendieck spectral sequence gives
$$E_2^{i,j} = \mathrm{Ext}^j(L^{-i}\mu^*\Omega_{\bar{F}}^1, \omega_F) \Rightarrow \mathrm{Ext}^{i+j}(L\mu^*\Omega_{\bar{F}}^1, \omega_F).$$
Observe that $E_2^{i,j} = \mathrm{Ext}^j(L^{-i}\mu^*\Omega_{\bar{F}}^1, \omega_F) = 0$ if $i>0, j = 0$ or $i<0$ or $j<0$. We conclude that
$$E_{\infty}^{0,1} \cong Ext^1(\mu^*\Omega_{\bar{F}}^1, \omega_F)~\mathrm{and}~E_{\infty}^{1,0} \cong 0,$$
hence
$$\mathrm{Ext}^1(L\mu^*\Omega_{\bar{F}}^1, \omega_F) \cong \mathrm{Ext}^1(\mu^*\Omega_{\bar{F}}^1, \omega_F) \cong \mathrm{Ext}^1(\Omega_{\bar{F}}^1, \omega_{\bar{F}}).$$

\textbf{Step 3:}
By $H^0(F, \omega_{F}) \cong H^0({\bar{F}},\omega_{\bar{F}})$, we have the following commutative diagram
\[\begin{CD}
H^1(F, T_{F})\cong \mathrm{Ext}^1(\Omega_F^1, \mathcal{O}_F)     @>\delta >>    \mathrm{Hom}(H^0(F, \Omega^n_{F}), H^1(F, \Omega^{n-1}_{F})) \\
@V\beta VV          @V\gamma VV               \\
\mathrm{Ext}^1(\Omega_{\bar{F}}^1, \mathcal{O}_{\bar{F}})    @>\alpha >>    \mathrm{Hom}(H^0({\bar{F}},\omega_{\bar{F}}), \mathrm{Ext}^1(\Omega_{\bar{F}}^1, \omega_{\bar{F}}))
\end{CD} \]
where $\gamma$ is induced by
\begin{equation}
\begin{split}
&H^1(F, \Omega^{n-1}_{F})\cong H^1(F, T_F \otimes \omega_{F}) \cong \mathrm{Ext}^1(\Omega_{F}^1, \omega_F) \rightarrow\\
& \rightarrow \mathrm{Ext}^1(\mu^*\Omega_{\bar{F}}^1, \omega_F) \cong \mathrm{Ext}^1(\Omega_{\bar{F}}^1, \omega_{\bar{F}}) \text{~where the arrow is induced by $\mu^*\Omega_{\bar{F}}^1 \rightarrow \Omega_{F}^1$}
\end{split}
\end{equation}
Then since the homomorphisms $\alpha$ and $\beta$ are injective, the period map $\delta$ is also injective.
\end{proof}

\begin{Remark}
The proof highly relies on the infinitesimal Torelli theorem, i.e., the injectivity of the period map $\alpha: \mathrm{Ext}^1(\Omega_{\bar{F}}^1, \mathcal{O}_{\bar{F}}) \rightarrow \mathrm{Hom}(H^0({\bar{F}},\omega_{\bar{F}}), \mathrm{Ext}^1(\Omega_{\bar{F}}^1, \omega_{\bar{F}}))$, so it does not work when the fibers have non-injective period maps (e.g. curves of genus 2). For the injectivity of the period maps of certain irregular surfaces, we refer to \cite{Re}.
\end{Remark}

\section{Example}\label{eg}
In this section, to illustrate the necessity of AS(2), we construct a variety $X$, such that $\kappa(X) = 1$ and $\mathrm{alb}_X$ is fibred by elliptic curves, but $\dim V^0(\omega_X, \mathrm{alb}_X) = 0$.
Let
\begin{itemize}
\item
$E$ be an abelian curve, $a \in E$ a torsion point of order 2, $t_a$ the involution  translating by $a$ and $i$ the involution induced by $(-1)_{E}$;
\item
$C$ a curve of genus $\geq 2$ with an involution $\eta$ such that the quotient $C/\eta$ is an elliptic curve;
\item
$Z=E \times C \times E$, $\sigma =t_a \times \eta \times \mathrm{id}_{E}$ and $\tau = i \times \mathrm{id}_C \times t_a$ two involutions on $Z$.
\end{itemize}
Since $\sigma$ and $\tau$ commute, the group $G=<\sigma, \tau> \cong \mathbb{Z}_2 \times \mathbb{Z}_2$. Note that $G$ acts on $Z$ freely, so the quotient map $\pi: Z \rightarrow X:=Z/G$ is \'{e}tale. Calculate that $\kappa(X) = 1$, $q(X) =2$, and the natural map $X \rightarrow C/\eta \times E/t_a$ is connected hence coincides with the Albanese map $\mathrm{alb}_X$. Then we have the following commutative diagram
\[\begin{CD}
E \times C \times E     @>\pi >> X \\
@Vp_2 \times p_3 VV                @V\mathrm{alb}_X VV       \\
C \times E    @> \pi_2 \times \pi_3 >>  C/\eta\times E/t_a
\end{CD} \]
where $p_2$ and $p_3$ are the projections from $E \times C \times E $ to the $2^{\mathrm{nd}}$ and the $3^{\mathrm{rd}}$ factors, and $\pi_2, \pi_3$ are the quotient maps from $C, E$ to $C/\eta, E/t_a$ respectively.

We can give $\omega_Z$ a $G$--invariant structure such that $\pi_*^G\omega_Z \cong \omega_X$, and similarly give $\omega_C$ an $\eta$--invariant structure and $\omega_E$ $t_a$--invariant and $i$--invariant structure; for $\alpha \in \mathrm{Pic}^0(C/\eta)$ and $\beta \in \mathrm{Pic}^0(E/t_a)$, we give $\pi_2^*\alpha$ an $\eta$--invariant structure such that $(\pi_2)_*^\eta (\pi_2^*\alpha) \cong \alpha$ and $\pi_3^*\beta$ a $t_a$--invariant structure such that $(\pi_3)_*^{t_a}(\pi_3^*\beta) \cong \beta$. Then we have $\alpha \boxtimes \beta \in \mathrm{Pic}^0(X) \cong \mathrm{Pic}^0(C/\eta) \times \mathrm{Pic}^0(E/t_a)$, and naturally $\omega_Z \otimes \pi^*(\alpha \boxtimes \beta)$ is a $G$--invariant sheaf such that
$$\pi_*^G(\omega_Z \otimes \pi^*(\alpha \boxtimes \beta)) \cong \omega_X \otimes \alpha \boxtimes \beta.$$

For $\alpha \boxtimes \beta \in \mathrm{Pic}^0(X) \cong \mathrm{Pic}^0(C/\eta) \times \mathrm{Pic}^0(E/t_a)$, we have
\begin{equation}
\begin{split}
& H^0(X, \omega_X \otimes \alpha \boxtimes \beta) \\
\cong &H^0(Z, \omega_Z \otimes \pi^*(\alpha \boxtimes \beta))^G \\
\cong &(H^0(E, \omega_E) \otimes H^0(C, \omega_C \otimes \pi_2^*\alpha) \otimes H^0(E, \omega_E \otimes \pi_3^*\beta))^G \\
\cong &(H^0(E, \omega_E) \otimes H^0(C, \omega_C \otimes \pi_2^*\alpha) \otimes H^0(E, \omega_E \otimes \pi_3^*\beta))^\sigma \cap \\
&(H^0(E, \omega_E) \otimes H^0(C, \omega_C \otimes \pi_2^*\alpha) \otimes H^0(E, \omega_E \otimes \pi_3^*\beta))^\tau \\
\cong &(H^0(E, \omega_E) \otimes H^0(C, \omega_C \otimes \pi_2^*\alpha)^\eta \otimes H^0(E, \omega_E \otimes \pi_3^*\beta)) \cap \\
&(H^0(E, \omega_E) \otimes H^0(C, \omega_C \otimes \pi_2^*\alpha) \otimes H^0(E, \omega_E \otimes \pi_3^*\beta)^{t_a-}) \\
\cong &H^0(E, \omega_E) \otimes H^0(C, \omega_C \otimes \pi_2^*\alpha)^\eta \otimes H^0(E, \omega_E \otimes \pi_3^*\beta)^{t_a-}
\end{split}
\end{equation}
where $H^0(E, \omega_E \otimes \pi^*\beta)^{t_a-}$ is the $t_a$--anti-invariant subspace, and the $4^{\mathrm{th}}$ ``$\cong$'' is due to that $H^0(E, \omega_E)$ is $t_a$--invariant and is $i$--anti-invariant.

Note that $H^0(C, \omega_C \otimes \pi_2^*\alpha)^\eta \cong H^0(C/\eta, \omega_{C/\eta} \otimes \alpha)$; and since $(\pi_3)_*(\omega_E) \cong \omega_{E/t_a} \oplus L$ where $L$ is a torsion line bundle on $E/t_a$ of order 2, so  $(\pi_3)_*(\omega_E \otimes \pi_3^*\beta) \cong \omega_{E/t_a}\otimes \beta \oplus L \otimes \beta$ where $\omega_{E/t_a}\otimes \beta$ is $t_a$--invariant and $L \otimes \beta$ is $t_a$--anti-invariant.
So we conclude that $H^0(X, \omega_X \otimes \alpha \boxtimes \beta)  \cong H^0(E, \omega_E) \otimes H^0(C/\eta, \omega_{C/\eta} \otimes \alpha)  \otimes H^0(E/t_a, L \otimes \beta) \neq 0$ if and only if $\alpha = \mathcal{O}_{C/\eta}$ and $\beta = L$. Therefore, $V^0(\omega_X) = \{\mathcal{O}_{C/\eta} \boxtimes L\}$.

\end{document}